\documentclass[12pt]{amsart}
\usepackage{geometry}                
\usepackage{
	bm,
	tikz-cd,
    dsfont,
	epstopdf,
	mathrsfs,
	mathabx,
	multirow,
	mathtools,
	dsfont,
	tabu,
	float,
	subfig,
	hyperref}
\usetikzlibrary{matrix}
\hypersetup{
    colorlinks=true,
    allcolors=blue
}

\usepackage{amsmath,amssymb,latexsym}
\usepackage{amsfonts,amstext,amscd}
\usepackage{amsopn, amsthm, amsgen}
\usepackage[english]{babel}
\usepackage{fancybox}
\usepackage[all]{xy}
\usepackage{xcolor}
\usepackage{verbatim}
\usepackage{graphicx}
\usepackage{enumerate}
\theoremstyle{plain}

\newtheorem{theorem}{Theorem}[section]
\newtheorem{corollary}[theorem]{Corollary}
\newtheorem{definition}[theorem]{Definition}
\newtheorem{example}[theorem]{Example}

\newtheorem{proposition}[theorem]{Proposition}
\newtheorem{remark}[theorem]{Remark}
\newtheorem{question}[theorem]{Question}
\numberwithin{equation}{section}

\newcommand{\A}{\mathbb{A}}

\newcommand{\Cont}{\mathrm{C}}

\newcommand{\Exp}{\mathrm{Exp}}
\newcommand{\id}{\mathrm{id}}
\newcommand{\Ll}{\mathcal{L}}
\newcommand{\N}{\mathbb{N}}

\newcommand{\Pp}{\mathcal{P}_f(\Ss)}
\newcommand{\Q}{\mathbb{Q}}
\newcommand{\R}{\mathbb R}
\newcommand{\Ss}{\mathsf{S}}

\newcommand{\Susp}{\Sigma_f(\Ss)}
\newcommand{\UC}{\mathbb{S}^1}
\newcommand{\Z}{\mathbb{Z}}
\newcommand{\T}{\mathbb{T}}
\newcommand{\Zz}{\widehat{\Z}}

\newcommand{\Char}{\mathrm{Char}}
\newcommand{\Hom}{\mathrm{Hom}}
\newcommand{\Homeo}{\mathrm{Homeo}}

\newcommand{\mTo}{\longmapsto}
\newcommand{\To}{\longrightarrow}
\newcommand{\ie}{\emph{i.e.\,}}
\date{\today}

\begin{document}

\title[Poincar\'e Theory for $\A/\Q$]{Poincar\'e Theory for the Ad\`ele Class Group $\A/\Q$ and Compact Abelian one dimensional Solenoidal Groups}
\author[M. Cruz -- L\'opez and A. Verjovsky]{Manuel Cruz -- L\'opez$^*$ and Alberto Verjovsky$^{**}$}

\address{$*$ Departamento de Matem\'aticas,
Universidad de Guanajuato, Jalisco s/n, Mineral de Valenciana,
Guanajuato, Gto. 36240 M\'exico.}
\email{manuelcl@ugto.mx}

\address{$**$ Instituto de Matem\'aticas, Unidad Cuernavaca,
Universidad Nacional Aut\'onoma de M\'exico, Apdo. Postal 2 C.P.
2000, Cuernavaca, Mor. M\'exico}
\email{alberto@matcuer.unam.mx}



\begin{abstract}

This article presents a generalization of the notion of \emph{Poincar\'e rotation set} to homeomorphisms of the ad\`ele class group $\A/\Q$ of the rational numbers $\Q$, which is a connected compact abelian group which can be identified with the one-dimensional universal solenoid $\Ss$,
\ie the algebraic universal covering of the circle.  The definition is first introduced in general for homeomorphisms of $\Ss$ which are isotopic to a translation, and then specializing in homeomorphisms of $\Ss$ isotopic to the identity, in which case the rotation set is a closed interval contained in the base leaf (the connected component of the identity). If in the latter case  the rotation interval reduces to a single 
element $\rho$ and $\rho$ is irrational (\ie it is a monothetic generator of $\Ss$), we show that the homeomorphism is semiconjugate to the translation $z\mapsto\rho{z}$, like in the classical theory of Poincar\'e.
This theory is valid for any general compact abelian one dimensional solenoidal group 
$\Ss_G$, which are Pontryagin duals of dense subgroups $G$ of the rational numbers with the discrete topology. These solenoidal groups are one-dimensional laminations which are locally homeomorphic to the product of a Cantor set by an interval so they behave very much like a ``diffuse'' version of the circle. Our approach differs from others because we use
Pontryagin duality of compact abelian groups to define the rotation sets.

\end{abstract}

\maketitle

\section[Introduction]{Introduction}
\label{introduction}

In 1885, H. Poincar\'e (see \cite{Poi}) introduced an invariant of topological conjugation for homeomorphisms of the unit circle which are isotopic to the identity:
\[ \rho:\Homeo_+(\UC)\To \UC, \quad f\mTo \rho(f), \]
called the \textsf{rotation number} of $f$. He then proved a remarkable topological classification theorem for the dynamics of any orientation preserving homeomorphism 
$f\in \Homeo_+(\UC)$: $f$ has a periodic orbit if and only $\rho(f)=e^{2\pi{i}\frac{p}{q}}, p,q\in\Z$ (\ie the rotation number is rational). If the rotation number is of the form
$\rho(f)=e^{2\pi{i}\alpha}$ with $\alpha$ irrational then $f$ is semiconjugate to the irrational rotation $R_{\rho(f)}\,$,
$z\mapsto{\rho(f)z}$. The semiconjugacy is actually a conjugacy if the orbits of $f$ are dense. This work was continued by A. Denjoy in 1932 (see \cite{Den}) who, among other things, showed that if $f$ is a diffeomorphism with irrational rotation number and its derivative has bounded variation then $f$ is conjugated to the rotation $R_{\rho(f)}$. In $1965$, V.I. Arnold (see \cite{Arn1}) solved the conjugation problem when the diffeomorphism is real analytic and close to a rotation, by introducing a Diophantine condition. An important issue is the existence of a differentiable conjugacy where M. Herman made so many important contributions (see \cite{Her}).

Further developments of this theory have been one of the most fruitful subjects in dynamical systems, as shown by the works of A.N. Kolmogorov, V.I. Arnold, J. Moser, M.R. Herman, A.D. Brjuno, J.C. Yoccoz, among others (see \cite{Kol},\cite{Arn2},\cite{Mos},\cite{Brj1,Brj2}, \cite{Yoc}; see also \cite{Ghys}, \cite{Her}, \cite{Nav}).

In general, for compact connected abelian groups there is a set called the \emph{rotation set} attached to a homeomorphism (see for instance, \cite{MZ} and \cite{Kwa}). For the case studied here of homeomorphisms isotopic to the identity of one dimensional compact solenoidal groups, this rotation set is an interval. 

This article generalizes the Poincar\'e theory to any compact abelian one dimensional solenoidal group $\Ss_G$, in the case when the \emph{rotation set consists of one point and this element is a monothetic generator of $\Ss_G$}. The theory is first developed for the general case of homeomorphisms of $\Ss_G$ which are isotopic to a translation by an element not in the base leaf. Translations by the zero element of the group leads to homeomorphisms isotopic to the identity which is precisely reminiscent of the classical theory. In this context, the semiconjugation theorem is proved.

Compact abelian one dimensional solenoidal groups $\Ss_G$ are obtained as continuous homomorphic images of the algebraic universal covering space of the circle 
\[ \displaystyle{\Ss := \varprojlim_{n\in \N} \R/n\Z} \] 
or, by Pontryagin duality, as compact abelian topological groups whose groups of characters are additive subgroups of the rational numbers with the discrete topology. In the case of the one dimensional universal solenoidal group $\Ss$, the group of characters is the whole discrete group $\Q$. 

The group $\Ss$ can be thought of as a generalized circle and it is isomorphic as a topological group to the \textsf{ad\`ele class group} of the rational numbers $\A/\Q$, which is the orbit space of the locally trivial $\Q$ -- bundle structure 
$\Q\hookrightarrow \A \To \A/\Q$, where $\A$ is the ad\`ele group of the rational numbers and $\Q\hookrightarrow \A$ is a discrete cocompact subgroup of $\A$. The ad\`ele class group is a fundamental arithmetic object in mathematics whose multiplicative part was invented by C. Chevalley for the purposes of simplifying and clarifying class field theory. This compact abelian group plays an essential role in the thesis of J. Tate (see \cite{Tat}) which laid the foundations for the Langlands program. 

By definition, $\Ss$ is a compact abelian topological group with a locally trivial $\Zz$ -- bundle structure $\Zz\hookrightarrow \Ss\To \UC$ and also a one dimensional foliated space whose leaves have a canonical affine structure isomorphic to the real one dimensional affine space $\mathbf{A}^1$. Here, $\displaystyle{\Zz := \varprojlim_{n\in \N} \Z/n\Z}$ is the profinite completion of the integers, and it is an abelian Cantor group. Thus, topologically, 
$\Ss$ is a compact and connected locally trivial fibration over the circle with fiber the Cantor set. 

More general objects are the so called solenoidal manifolds, which were introduced by Dennis Sullivan (see \cite{Sul} and \cite{Ver}). These solenoidal manifolds are Polish spaces with the property that each point has a neighborhood which is homeomorphic to an open interval times a Cantor set. He shows that any compact one dimensional \emph{orientable} solenoidal manifold is the suspension of a homeomorphism of the Cantor set. Examples of one dimensional solenoidal manifolds are one dimensional tiling spaces and one dimensional quasicrystals like the ones studied by R.F. Williams and L. Sadun, and also by J. Aliste -- Prieto (see \cite{WS} and \cite{Ali}). 

\begin{remark} 
In principle, Poincar\'e theory might be described for general compact, orientable one dimensional solenoidal manifolds. What makes the difference in our case is the fact that we can apply to these groups the Pontryagin duality and the classical theory of harmonic analysis for compact and locally compact Abelian groups.
\end{remark}

In the development of the article, the theory is first described for the ad\`ele class group of the rational numbers $\Ss$, since this is the paradigmatic example and all the ideas are already present there. This is done by using the notion of asymptotic cycles of S. Schwartzman (see \cite{Sch}). So the analysis is first done for the case of homeomorphisms of solenoids which are isotopic to the identity. After that it is treated the more general case of homeomorphisms isotopic to translations with the translation element not in the base leaf. 

We now briefly describe the definition of the rotation set and state the main theorem (see Sections \ref{rotation_set} and \ref{Poincare_theory}).  

Suppose that $f:\Ss\To \Ss$ is any homeomorphism isotopic to the identity which can be written as $f=\id + \varphi$, where $\varphi:\Ss\To \Ss$ is the displacement function along the one dimensional leaves of $\Ss$ with respect to the affine structure. The suspension of 
$f$ is the compact space
\[ \Susp := \Ss\times [0,1] /(z,1)\sim (f(z),0). \] 

Since $f$ is isotopic to the identity, it follows that $\Susp$ is homeomorphic to the product space $\Ss\times \UC$ which is a compact abelian topological group. The space $\Susp$ has a natural compact abelian group structure described very detailed in Section \ref{rotation_translation}. For the sake of simplicity, identify the product group structure of $\Ss \times \UC$ with the group structure on $\Susp$. So the character group of the suspension of $f$ is
$$
\Char(\Susp)\cong \Char(\Ss\times \UC)\cong \Char(\Ss)\times \Char(\UC)\cong \Q\times \Z.
$$

The associated suspension flow $\phi_t:\Susp\To\Susp$ is given by:
\[ \phi_t(z,s) = (f^m(z),t+s-m),\qquad (m\leq t+s < m+1). \]

Now, for any given character $\chi_{q,n}\in \Char(\Susp)$, there exists a unique 1 -- cocycle 
\[ C_{\chi_{q,n}}:\R\times \Susp\To \R \] 
associated to $\chi_{q,n}$ (see Section \ref{rotation_set} for complete information) such that
\[ \chi_{q,n}(\phi_t(z,s)) = \exp(2\pi iC_{\chi_{q,n}}(t,(z,s)))\cdot \chi_{q,n}(z,s), \] 
for every $(z,s)\in \Susp$ and $t\in \R$. Using the definition of $\chi_{q,n}$ and $\phi_t$ and comparing terms in the last equality one obtains an explicit expression for the 1 -- cocycle $C_{\chi_{q,n}}(t,(z,s))$. By Birkhoff's ergodic theorem, there is a well defined homomorphism
\[ H_f:\Char(\Susp)\To \R \] 
given by
\[ H_f(\chi_{q,n}) = \int_{\Susp} C_{\chi_{q,n}}(1,(z,s)) d\nu, \] 
where $\nu$ is a $\phi_t$--invariant Borel probability measure on $\Susp$. The explicit calculation of the 1 -- cocycle implies that the last integral only depends on an $f$ -- invariant Borel probability measure $\mu$. So, if $\Pp$ is the weak$^*$ compact convex space consisting of all such measure on $\Ss$, the well defined continuous homomorphism
\[ \rho_\mu(f):\Char(\Susp)\To \UC \] 
given by
\[ \rho_\mu(f)(\chi_{q,n}) := \exp(2\pi iH_f(\chi_{q,n})) \] 
determines an element in $\Char(\Char(\Susp))\cong \Ss\times \UC$ not depending on the second component. By Pontryagin's duality theorem, it defines an element 
$\rho_\mu(f)\in \Ss$ called the rotation element associated to $f$, which is the generalized Poincar\'e rotation number.

As expected, $\rho_\mu(f)$ is an element in the solenoid itself and it measures, in some sense, the average displacement of points under iteration of $f$ along the one dimensional leaves with the Euclidean metric.

If $\rho:\Pp\To \Ss$ is the map given by $\mu\mTo \rho_\mu(f)$, then $\rho$ is continuous from $\Pp$ to $\Ss$. Since $f$ is isotopic to the identity, the image $\rho(\Pp)$ is a compact subset of $\Ss$. The pathwise component of $\Ss$ are one dimensional leaves, so the set 
$\rho(\Pp)$ is a compact interval $I_f$, which, up to a translation, is contained in the one parameter subgroup $\Ll_0$. Now, $\Ll_0$ is canonically isomorphic to $\R$, so it is possible to identify $I_f$ with an interval in the real line. In particular, if $f$ is uniquely ergodic, then the interval $I_f$ reduces to a point and this case. \\

\noindent \textbf{Definition \ref{pseudoirrational_rotation}}
The homeomorphisms $f$ is a \textsf{pseudoirrational rotation} if $I_f$ consists of a single point $I_f=\{\alpha\}$. In this case, and only in this case, we call $\alpha$ the \emph{rotation element} of $f$. \\

In order to state the main result, some definitions are required. \\

First observe that, since $\Ss$ is torsion free, it follows that there does not exist a notion of ``rational'' and so it is only required to give a suitable definition of what ``irrational'' is. This goes as follows: \\

\noindent \textsf{Definition \ref{irrational_element}} 
The $\alpha\in \Ss$ is \textsf{irrational} if $\{n\alpha:n\in \Z\}$ is dense in $\Ss$. In classical terminology, $\Ss$ is said to be \textsf{monothetic} with generator $\alpha$. \\

Fixing a measure $\mu\in \Pp$ determines a rotation element $\rho_\mu(f)$ of $f$, simply denoted by $\rho(f)$. If $F$ is any lift of $f$ of the form
\[ F(t,k)=\left( F_k(t),k \right)\qquad (t,k)\in \R\times \Zz, \]
where $F_k:\R\To \R$ is a homeomorphism which depends continuously on $k\in \Zz$, the following definition seems appropriate (see Section \ref{Poincare_theory} for details). \\

\noindent \textbf{Definition \ref{bounded_meanvariation}} 
The homeomorphism $f$ has \textsf{bounded mean variation} if there exists $C>0$ such that the sequence $\{F_k^n(t) - t - n\tau(F)\}_{n\geq 1}$ is uniformly bounded by $C$, i.e. $C$ is independent of $(t,k)$. Here, $F_k^n$ is any lift of $f^n$ and $\tau(F)$ is a lifting of $\rho(f)$ to $\R\times \Zz$ contained in the same leaf as $(t,k)$. \\

The generalized semiconjugacy theorem can be stated as follows: \\

\noindent \textbf{Theorem \ref{Poincare_theorem}} 
If $f\in \Homeo_+(\Ss)$ is a pseudoirrational rotation with irrational rotation element 
$\rho(f)$, then $f$ is semiconjugated to the irrational rotation $R_{\rho(f)}$ if and only if $f$ has bounded mean variation. \\

The conjugacy question remains: \\

\noindent \textbf{Question \ref{question_conjugacy}} 
Under the same hypothesis of this theorem, is $f$ conjugated to the rotation $R_{\rho(f)}$ when $f$ is minimal? \\

The answer to this question and more dynamical results is the subject of recent investigation (see \cite{CV}). \\

For general one dimensional compact abelian solenoidal group $\Ss_G$, this theorem can be stated as:\\

\noindent \textbf{Theorem \ref{solenoidal_Poincare-theorem}}
Suppose that $f:\Ss_G\To \Ss_G$ is any homeomorphism isotopic to the identity, or isotopic to a rotation by an element not in the base leaf, with irrational rotation element $\rho(f)$. $f$ is semiconjugated to the irrational rotation $R_{\rho(f)}$ if and only if $f$ has bounded mean variation. \\

Closed related to this work is the article by J. Kwapisz (see \cite{Kwa}) who gives a definition of a rotation element for homeomorphisms of the real line with almost periodic displacement. When the displacement is limit periodic, the corresponding convex hull is a compact abelian one dimensional solenoidal group. However, we consider $\Ss$ as a ``generalized circle'' (i.e. a compact abelian one dimensional pro -- Lie group) and we develop the theory from this perspective.  

Other similar studies of the Poincar\'e theory have been developed very recently by several authors. In the article \cite{Jag}, T. J\"ager proved that a minimal homeomorphism of the 
$d$ -- dimensional torus is semiconjugated to an irrational rotation if and only if it is a pseudoirrational rotation with bounded mean motion (see also \cite{AJ} and \cite{BCJL}).
\smallskip

The article is organized as follows. Section \ref{universal_solenoid} defines the algebraic universal covering space of the circle, its character group, the suspension of a homeomorphism isotopic to the identity and its corresponding character group. Section \ref{rotation_set} introduces the notion of 1--cocycle and gives the definition of the generalized rotation set and element. In order to define this generalized rotation element 
$\rho(f)$ it is necessary to use the following ingredients: Pontryagin's duality theory for compact Abelian groups, the Bruschlinsky--Eilenberg homology theory and Schwartzman theory of asymptotic cycles as well as the notion of a 1--cocycle and ergodic theory. The semiconjugacy theorem is proved in Section \ref{Poincare_theory}. Finally, Section \ref{rotation_translation} introduces a general definition for the case of homeomorphisms isotopic to translations whose rotation element is not in the base leaf and for homeomorphisms of general compact abelian one dimensional solenoidal groups.

\section[The algebraic universal covering space of the circle]{The algebraic universal covering space of the circle}
\label{universal_solenoid}

This Section introduces the algebraic universal covering space of the circle, its character group, the suspension of a homeomorphism isotopic to the identity and its corresponding character group.

\subsection[The universal one dimensional solenoid]{The universal one dimensional solenoid}

It is well known, by covering space theory, that for any integer $n\geq 1$, it is defined the unbranched covering space of degree $n$, $p_n:\UC \To \UC$ given by $z\mTo {z^n}$. If $n,m\in \Z^+$ and $n$ divides $m$ then there exists a unique covering map 
$p_{nm}:\UC\To \UC$ such that $p_n \circ p_{nm} = p_m$. This determines a projective system of covering spaces $\{\UC,p_n\}_{n\geq 1}$ whose projective limit is the \textsf{universal one dimensional solenoid}
\[ \Ss := \varprojlim_{n\in \N} \{\UC,p_n\} \] 
with canonical projection $\Ss\To \UC$ determined by projection onto the first coordinate, with a locally trivial $\Zz$--bundle structure $\Zz\hookrightarrow \Ss \To\UC$. Here, 
$\displaystyle{\Zz := \varprojlim_{n\in \N} \Z/n\Z}$ is the profinite completion of $\Z$, which is a compact, perfect and totally disconnected Abelian topological group homeomorphic to the Cantor set. Being $\Zz$ the profinite completion of $\Z$, it admits a canonical inclusion of $\Z$ whose image is dense.

$\Ss$ can also be realized as the orbit space of the $\Q$--bundle structure 
$\Q \hookrightarrow \A \To \A/\Q$, where $\A$ is the ad\`ele group of the rational numbers which is a locally compact Abelian group, $\Q$ is a discrete subgroup of $\A$ and $\A/\Q \cong \Ss$ is a compact Abelian group (see \cite{RV}). From this perspective, 
$\A/\Q$ can be seen as a projective limit whose $n$--th component corresponds to the unique covering space of degree $n\geq 1$ of $\UC$. $\Ss$ is also called the \textsf{algebraic universal covering space} of the circle $\UC$. The Galois group of the covering is $\Zz$, the \textsf{algebraic fundamental group} of $\UC$.

By considering the properly discontinuously free action of $\Z$ on $\R\times \Zz$  given by
\[ \gamma\cdot (t,k) := (t+\gamma,k-\gamma) \quad (\gamma\in \Z), \]
$\Ss$ is identified with the orbit space $\R\times_{\Z} \Zz$. Here, $\Z$ is acting on $\R$ by covering transformations and on $\Zz$ by translations. The path connected component of the identity element $0\in \Ss$ is called the \textsf{base leaf} and will be denoted by 
$\Ll_0$. Clearly, $\Ll_0$ is the image of $\R\times \{0\}$ under the canonical projection 
$\R\times \Zz\To \Ss$ and it is homeomorphic to $\R$.

In summary, $\Ss$ is a compact, connected, Abelian topological group and also a one dimensional lamination where each ``leaf" is a simply connected one dimensional manifold, homeomorphic to the universal covering space $\R$ of $\UC$, and a typical ``transversal" is isomorphic to the Cantor group $\Zz$. $\Ss$ also has a leafwise 
$\Cont^\infty$ Riemannian metric (i.e., $\Cont^\infty$ along the leaves) which renders each leaf isometric to the real line with its standard metric. So, it makes sense to speak of a rigid translation along the leaves. The leaves also have a natural order equivalent to the order of the real line.

\subsubsection*[Characters of $\Ss$]{Characters of $\Ss$}
\label{charactersS}

Denote by $\Char(\Ss):=\Hom_{\text{cont}}(\Ss,\UC)$ the topological group endowed with the uniform topology consisting of all continuous homomorphisms from $\Ss$ into the multiplicative group $\UC$. This group is called the \textsf{Pontryagin dual} of $\Ss$ or, the \textsf{Character group} of $\Ss$. From what has been said before, $\Ss\cong \A/\Q$ and, since $\A$ is selfdual (i.e., $\A\cong \Char(\A)$), it follows that $\Char(\Ss)\cong \Q$. If $\check{H}^1(\Ss,\Z)$ denotes the first $\mathrm{\check{C}}$ech cohomology group of $\Ss$ with coefficients in $\Z$, since $\Ss$ is a compact connected Abelian group,
$\check{H}^1(\Ss,\Z)\cong \Char(\Ss)$ (See \cite{Ste}).

If $\chi:\Ss\To \UC$ is any character, $\chi$ is completely determined by its values when restricted to the dense one parameter subgroup $\Ll_0$. Since $\Ll_0$ is canonically isomorphic to the additive group $(\R,+)$, the restriction of $\chi$ to $\Ll_0$ is of the form $t\mTo \exp(2\pi its)$. Since 
$\Char(\Ss)\cong \Q$, $s$ must be rational. Now, given any $z\in \Ss$, there exists an element $n\in \Zz \subset \Ss$ (which can be approximated by integers) such that 
$z+n \in \Ll_0$. By continuity, the value of the character $\chi$ at $z$ is 
\[ \chi(z) = \exp(2\pi iq(z+n)) = \exp(2\pi iqz) \] 
for any $q\in \Q$. In this situation we will write $\chi(z) = \Exp{(2\pi iqz)}$.

\subsubsection*[Homeomorphisms of $\Ss$]{Homeomorphisms of $\Ss$}

We only consider the group which consists of all homeomorphisms of $\Ss$ which are isotopic to the identity and can be written as $f = \id + \varphi$, where 
$\varphi:\Ss\To \Ss$ is given by $\varphi(t)=f(t)-t$ and describes the displacement of points $t\in \Ss$ along the leaf containing it. The symbol ``--'' refers to the additive group operation in the solenoid. Denote the set of all such functions $\varphi$ by 
$\Cont_+(\Ss)$.

Since $f$ is a homeomorphism preserving the order in the leaves, it follows that there is a one to one correspondence between $\Cont_+(\Ss)$ and the set of real valued continuous functions with the property that if $x$ and $y$ are in the same one dimensional leaf and if $t<s$, then $t+\varphi(t)<s+\varphi(s)$. Therefore, $\Cont_+(\Ss)$ can be identified with the Banach space of real valued continuous functions $\Cont(\Ss,\R)$. 

As mentioned in the Introduction, the solenoid has a leafwise $\Cont^\infty$ Riemannian metric (i.e., $\Cont^\infty$ along the leaves) which renders each leaf isometric to the real line with its standard metric. Hence, the displacement function $\varphi$ can be thought of as a continuous real valued function which we denote with the same symbol $\varphi$. In fact, since every leaf $\Ll\subset \Ss$ is dense, the restriction of this function to $\Ll$, denoted by $\varphi_{\Ll}$, completely determines the function. Furthermore, 
$\varphi_{\Ll}$ is an almost periodic function whose convex hull is the solenoid and thus, 
$\varphi_{\Ll}$ is a limit periodic function (see \cite{Pon}).

\begin{remark}
Denote by $\Homeo_+(\Ss)$ the group of all homeomorphisms $f:\Ss\To \Ss$ which are isotopic to the identity and can be written as $f = \id + \varphi$, with 
$\varphi\in \Cont_+(\Ss)$; i.e.,
\[ \Homeo_+(\Ss) = \{f\in \Homeo(\Ss) : f = \id + \varphi, \; \varphi\in \Cont_+(\Ss)\}. \]
\end{remark}

\subsection[The suspension of a homeomorphism]{The suspension of a homeomorphism}
\label{suspension_homeomorphism}

Let $f:\Ss\To \Ss$ be any homeomorphism isotopic to the identity. On $\Ss\times \R$ define the $\Z$--action:
\[ \Z\times (\Ss\times \R) \To \Ss\times \R, \qquad (n,(z,t))\mTo (f^n(z),t+n). \]

As usual, being in the same orbit defines an equivalence relation in the space 
$\Ss\times \R$ and if $(z,t)$ is any point in $\Ss\times \R$, denote by 
\[ [(z,t)] = \{ (f^n(z),t+n) : n\in \Z \} = \{ (z+n\alpha,t+n) : n\in \Z \} \] 
the equivalence class of the point, i.e. its $\Z$--orbit. The orbit space $\Susp$ of this action is the \textsf{suspension} of $f$.

Take any two points $[(z,t)]$ and $[(w,s)]$ in $\Susp$ and define the following operation:
\[ [(z,t)]\cdot [(w,s)] :=  [(z+w+(n+m)\alpha,t+s+(n+m))] = [(z+w,t+s)]. \]

Since $f$ is isotopic to the identity, it follows that $\Susp$ is homeomorphic to the product space $\Ss\times \UC$, which is a compact Abelian topological group. In $\Susp$ there is a well defined flow 
\[ \phi:\R\times \Susp\To \Susp \] 
called the \textsf{suspension flow} of $f$, given by
\[ \phi(t,(z,s)) = (f^m(z),t+s-m), \] 
if $m\leq t+s \leq m+1$. The canonical projection $\pi:\Ss\times [0,1]\To \Susp$ sends 
$\Ss\times \{0\}$ homeomorphically onto its image $\pi(\Ss\times \{0\})\equiv \Ss$ and every orbit of the suspension flow intersects $\Ss$. The orbit of any $(z,0)\in \Susp$ must coincide with the orbit $\phi_t(z,0)$ at time $0\leq t\leq T$ for $T$ an integer.

\subsubsection*[Characters of the suspension]{Characters of the suspension}

Denote by $\Cont(\Susp,\UC)$ the topological space which consists of all continuous functions defined on $\Susp$ with values in the unit circle $\UC$ with the topology of uniform convergence on compact sets (i.e., the compact and open topology). Clearly, this is an Abelian topological group under pointwise multiplication. The subset 
$R(\Susp,\UC)\subset \Cont(\Susp,\UC)$ which consists of continuous functions 
$h:\Susp\To \UC$ that can be written as $h(z,s) = \exp(2\pi i\psi(z,s))$ with 
$\psi:\Susp\To \R$ a continuous function, is a closed subgroup. Hence, the quotient group 
$\Cont(\Susp,\UC)/R(\Susp,\UC)$ is a topological group. By Bruschlinsky--Eilenberg's theory (see\cite{Sch}), it is known that
\[ \check{H}^1(\Susp,\Z)\cong \Cont(\Susp,\UC)/R(\Susp,\UC). \]

Since
\[ \check{H}^1(\Susp,\Z)\cong \Char(\Susp), \] 
it follows that
\[ \Char(\Susp)\cong \Cont(\Susp,\UC)/R(\Susp,\UC). \]

On the other hand, using the algebraic structure of the product group $\Ss\times \UC$, the character group of the suspension is given by
\[ \Char(\Susp)\cong \Char(\Ss)\times \Char(\UC)\cong \Q\times \Z. \]

According with the definition of $\Exp$ in Subsection \ref{charactersS}, given any element 
$(q,n)\in \Q\times \Z$, the corresponding character $\chi_{q,n}\in \Char(\Susp)$ can be written as
\begin{align*}
\chi_{q,n}(z,s)	&= \Exp(2\pi iqz)\cdot \exp(2\pi ins)\\
						&= \Exp(2\pi i(qz+ns)),
\end{align*}
for any $(z,s)\in \Susp$.

\subsubsection*[Measures]{Measures}

Given any $f$--invariant Borel probability measure $\mu$ on $\Ss$ and $\lambda$ the usual Lebesgue measure on $[0,1]$, the product measure $\mu\times \lambda$ leads to define a $\phi_t$--invariant Borel probability measure on $\Susp$. Reciprocally, given any $\phi_t$--invariant Borel probability measure $\nu$ on $\Susp$, it can be defined, by disintegration with respect to the fibers, an $f$--invariant Borel probability measure $\mu$ on $\Ss$. Denote by $\Pp$ the weak$^*$ compact convex space of $f$--invariant Borel probability measures defined on $\Ss$.

\section[The rotation set]{The rotation set}
\label{rotation_set}

This Section presents the notion of 1 -- cocycle and gives the definition of the generalized rotation set and element. In order to define this generalized rotation element  $\rho(f)$ associated to a homeomorphism $f:\Ss\To \Ss$ isotopic to the identity, it is necessary to use Pontryagin's duality theory for compact abelian groups, the Bruschlinsky -- Eilenberg homology theory and Schwartzman theory of asymptotic cycles as well as the notion of a 1--cocycle and ergodic theory.

\subsection[1--cocycles]{1--cocycles}

A 1--\textsf{cocycle} associated to the suspension flow $\phi_t$ is a continuous function
\[ C:\R\times \Susp\To \R \]
which satisfies the relation
\[ C(t+u,(z,s)) = C(u,\phi_t(z,s)) + C(t,(z,s)), \] 
for every $t,u\in \R$ and $(z,s)\in \Susp$. The set which consists of all 1--cocycles associated to 
$\phi_t$ is an Abelian group denoted by $\Cont^1(\phi)$. A 1--\textsf{coboundary} is the 1--cocyle determined by a continuous function $\psi:\Susp\To \R$ such that
\[ C(t,(z,s)):= \psi(z,s) - \psi(\phi_t(z,s)). \] 

The set of 1--coboundaries $\Gamma^1(\phi)$ is a subgroup of $\Cont^1(\phi)$ and the quotient group
\[ H^1(\phi):=\Cont^1(\phi)/\Gamma^1(\phi), \]
is called the 1--\textsf{cohomology group} associated to $\phi_t$. The proof of the next proposition (for an arbitrary compact metric space) can be seen in \cite{Ath}.

\begin{proposition}
\label{associated_cocycle} 
For every continuous function $h:\Susp\To \UC$ there exists a unique 1--cocycle 
$C_h:\R\times \Susp\To \R$ associated to $h$ such that
\[ h(\phi_t(z,s)) = \exp(2\pi iC_h(t,(z,s)))\cdot h(z,s), \] 
for every $(z,s)\in \Susp$ and $t\in \R$.
\end{proposition}

This proposition implies that there is a well defined homomorphism
\[ \Char(\Susp)\cong \check{H}^1(\Susp,\Z)\To H^1(\phi) \] 
by sending any character $\chi_{q,n}\in \Char(\Susp)$ to the cohomology class $[C_{\chi_{q,n}}]$, where $C_{\chi_{q,n}}$ is the unique 1--cocycle associated to $\chi_{q,n}$.

Applying the above proposition to any nontrivial character $\chi_{q,n}\in \Char(\Susp)$ the following relation is obtained:
\[ \chi_{q,n}(\phi_t(z,s)) = \exp(2\pi iC_{\chi_{q,n}}(t,(z,s)))\cdot \chi_{q,n}(z,s). \] 

Using the explicit expressions for the characters on both sides of the above equation, the next equalities hold
\begin{align*}
\chi_{q,n}(\phi_t(z,s))		&= \chi_{q,n}(f^m(z),t+s-m)\\
									&= \Exp(2\pi i(qf^m(z) + n(t+s-m)))\\
									&= \Exp(2\pi i(qf^m(z) + nt + ns))
\end{align*} and

\[ \chi_{q,n}(z,s) = \Exp(2\pi i(qz+ns)). \]

Comparing these two expressions one gets
\begin{equation}
\label{cocycle1} 
C_{\chi_{q,n}}(t,(z,s)) = q(f^m(z)-z) + nt.
\end{equation}

Now recall that $f:\Ss\To \Ss$ is a homeomorphism isotopic to the identity of the form 
$f=\id + \varphi$, where $\varphi:\Ss\To \Ss$ is the displacement function, where, as described before, $\varphi$ can also be considered as a real valued function on the solenoid. If $t=1$, then $m=1$ and the 1--cocycle at time $t=1$ is
\begin{equation}
\label{time1cocycle} 
C_{\chi_{q,n}}(1,(z,s)) = q\varphi(z) + n.
\end{equation}

\subsection[The rotation element]{The rotation element}

If $\nu$ is any $\phi_t$--invariant Borel probability measure on $\Susp$, by Birkhoff's ergodic theorem there is a well defined homomorphism $H^1(\phi)\To \R$ given by
\[ [C_\chi]\mTo \int_{\Susp} C_\chi(1,(z,s)) d\nu. \]

Now, composing the two homomorphisms
\[ \Char(\Susp)\To H^1(\phi)\To \R \] 
it is obtained a well defined homomorphism $H_{f,\nu}:\Char(\Susp)\To \R$ given by
\[ H_{f,\nu}(\chi_{q,n}) := \int_{\Susp} C_{\chi_{q,n}}(1,(z,s)) d\nu. \]

Denote by $\mu$ the $f$--invariant Borel probability measure on $\Ss$ obtained by disintegration of 
$\nu$ with respect to the fibers. Evaluating the above integral using equation (\ref{time1cocycle}) gives
\begin{align*}
H_{f,\nu}(\chi_{q,n})	&= \int_{\Susp} (q\varphi + n)d\nu\\
								&= q\int_\Ss \varphi d\mu + n.
\end{align*}

Hence $H_{f,\nu}$ determines an element in $\Hom(\Char(\Susp),\R)$ for each measure $\nu$ in 
$\Susp$, and therefore, for each measure $\mu\in \Pp$. Hence one gets a well defined function
\[ H_f:\Pp\To \Hom(\Char(\Susp),\R) \] 
given by $\mu\mTo H_{f,\mu}$, where $H_{f,\mu}$ is
\[ H_{f,\mu}(\chi_{q,n}) = q\int_\Ss \varphi d\mu + n. \]

By composing $H_f$ with the continuous homomorphism
\[ \Hom(\Char(\Susp),\R)\To \Char(\Char(\Susp)) \] 
given by
\[ H_{f,\mu}\mTo \pi\circ H_{f,\mu}, \] 
where $\pi:\R\To \UC$ is the universal covering projection, we obtain a well defined continuous function $\rho:\Pp\To \Char(\Char(\Susp))$ given by
\[ \mu\mTo \rho_\mu := \pi\circ H_{f,\mu}. \]

That is, for each $\mu\in \Pp$, there exists a well defined continuous homomorphism 
\[ \rho_\mu:\Char(\Susp)\To \UC \] 
given by
\begin{align*}
\rho_\mu(\chi_{q,n})	&:= \exp(2\pi iH_{f,\mu}(\chi_{q,n}))\\
								&= \exp \left(2\pi iq \int_\Ss \varphi d\mu\right).\\
\end{align*}

By Pontryagin's duality theorem,
\[ \Char(\Char(\Susp))\cong \Susp \]
and therefore $\rho_\mu\in \Susp$. Since $\Susp\cong \Ss\times \UC$ and 
$\rho_\mu(\chi_{q,n})=\rho_\mu(\chi_{q,0})$, it follows that $\rho_\mu$ does not depend on the second component and so, the identification $\rho_\mu = (\rho_\mu,1)\in \Ss\times \UC$ can be made. More precisely, it is well known that every nontrivial character of $\Char(\Ss)\cong \Q$ is of the form 
$\chi_a$ for some $a\in \A$ and the map $\A\To \Char(\Q)$ given by $a\mTo \chi_a$ induces an isomorphism $\Char(\Q)\cong \A/\Q\cong \Ss$. This produces a genuine element $\rho_\mu\in \Ss$.

\begin{definition}
The element $\rho_\mu(f) := \rho_\mu \in \Ss$ defined above is the \emph{\textsf{rotation element}} associated to $f$ with respect to the measure $\mu$.
\end{definition}

\begin{remark} 
By definition, $\rho_\mu(f)$ can be identified with the element $\int_\Ss \varphi d\mu$ in the solenoid $\Ss$ determined by the character of $\Q$ given by 
\[ q\mTo \Exp \left(2\pi iq\int_\Ss \varphi d\mu \right). \] 

That is, $\rho_\mu(f)$ is \emph{solenoid -- valued}.
\end{remark}

If $\rho:\Pp\To \Ss$ is the map given by $\mu\mTo \rho_\mu(f)$, then $\rho$ is continuous from $\Pp$ to $\Ss$. Since $\Pp$ is compact and convex, and $f$ is isotopic to the identity, the image $\rho(\Pp)$ is a compact subset of $\Ss$.

\begin{definition}
\label{rotationset}
The \emph{\textsf{rotation set}} of $f$ is $\rho(\Pp)$.
\end{definition}

Since the pathwise component of $\Ss$ are one dimensional leaves, this set $\rho(\Pp)$ is a compact interval $I_f$, which, up to a translation, is contained in the one parameter subgroup $\Ll_0$. Since $\Ll_0$ is canonically isomorphic to $\R$, it is possible to identify $I_f$ with an interval in the real line. In particular, if $f$ is uniquely ergodic, then the interval $I_f$ reduces to a point and the rotation element is a unique element of $\Ss$.

\begin{remark}
The rotation element of $f$ can be interpreted as the exponential of an asymptotic cycle, in the sense of Schwartzman, of the suspension flow $\{\phi_t\}_{t\in\R}$ of $f$ (see \cite{Sch}; see also \cite{AK}, \cite{Pol}). If $A_\nu\in \Hom(\check{H}^1(\Susp,\Z),\R)=\Hom(\Char(\Susp),\R)$ denotes the asymptotic cycle associated to the $\{\phi_t\}_{t\in\R}$--invariant measure $\nu$, then 
$\rho_\nu(f)=\exp(2\pi iA_\nu)$.
\end{remark}

\begin{remark}
From Birkhoff's ergodic theorem, for any ergodic $f$--invariant measure $\mu$,
$$ 
\int_\Ss \varphi d\mu = \underset{n\to\infty}\lim \ \frac{1}{n} \sum_{j=0}^n \varphi(f^j(z)),
$$
for $\mu$--almost every point $z\in\Ss$. We could have used this to define the rotation element with respect to an (ergodic) measure. Since we wanted to make explicit the role of the measure, we used the theory of asymptotic cycles in the sense of Schwartzman. (Compare \cite{Kwa}, Theorem 3.)
\end{remark}

\subsection[Basic example and properties]{Basic example and properties}

\subsubsection*[Basic example: Rotations]{Basic example: Rotations}

Let $\alpha$ be any element in $\Ll_0\subset \Ss$ and consider the rotation $R_\alpha:\Ss\To \Ss$ given by $z\mTo z + \alpha$. The suspension flow $\phi_t:\Ss\times \UC\To \Ss\times \UC$ is given by
\[ \phi_t(z,s) = (z+m\alpha,t+s-m), \] 
if $m\leq t+s< m+1$. If $\chi_{q,n}\in \Char(\Susp)$ is any nontrivial character then
\[ H_{R_\alpha,\mu}(\chi_{q,n}) = q\int_\Ss \alpha d\mu + n = q\alpha + n. \]
This implies that
\[ \rho_\mu(R_\alpha)(\chi_{q,n}) = \exp(2\pi i(q\alpha + n)) = \exp(2\pi iq\alpha) \] 
and $\rho_\mu(R_\alpha) = \alpha$.

\subsubsection*[Properties]{Properties}

\begin{enumerate}

\item \textsf{(Invariance under conjugation)} Let $f$ and $g$ be any two homeomorphisms isotopic to the identity and $h=\id+\psi$. If $h\circ f=g\circ h$ then $\rho_\mu(f)=\rho_\mu(g)$. In particular, if $f$ is conjugated to a rotation $R_\alpha$ then $\rho_\mu(f)=\alpha$.

\begin{proof} 
Observe first that $h\circ f=g\circ h$ implies that $h\circ f^m=g^m\circ h$ and 
$f^m + \psi\circ f^m=g^m\circ h - h + h$. That is
\[ f^m-\id = (g^m-\id)\circ h + \psi - \psi\circ f^m. \] 
Therefore the 1--cocycle associated to any nontrivial character $\chi_{q,n}$ at time $t=1$ has the form
\begin{align*}
C_{\chi_{q,n}}(1,(z,s))	&= q(f(z)-z) + n\\
									&= q[(g(h(z))-h(z)) + \psi(z) - \psi\circ f(z)] + n.
\end{align*} 
Since $\mu$ is both $f$ and $g$ invariant, we get

\begin{align*}
H_{f,\mu}(\chi_{q,n})	&= q\int_\Ss (f(z)-z)d\mu + n\\
									&= q\int_\Ss (g(h(z))-h(z))dh_*\mu + n\\
									&= H_{g,\mu}(\chi_{q,n}).
\end{align*} 
Hence, $\rho_\mu(f)=\rho_\mu(g)$.
\end{proof}

\item \textsf{(Continuity)} The function $\rho_\mu:\Homeo_+(\Ss)\To \Ss$ given by
\[ f=\id+\varphi\longmapsto \int_\Ss \varphi d\mu \] 
is continuous with respect to the uniform topology in $\Homeo_+(\Ss)$.

\item \textsf{(The rotation element is equal to zero if and only if $f$ has a fixed point)} Indeed, if $f$ has a fixed point $x$ then $\varphi(x)=0$; if $\mu=\delta_x$ is the Dirac mass at $x$ then 
$\int_\Ss \varphi d\mu=0$ and therefore $\rho_\mu(f)=0$. On the other hand, if $\rho_\mu(f)=0$ then $\int_\Ss \varphi d\mu=0$ and $\varphi$ must vanish at some point $x$ which must be a fixed point of $f$.

\end{enumerate}

\subsection[The rotation element \`a la de Rham]{The rotation element \`a la de Rham}

If $d\lambda$ denotes the usual Lebesgue measure on $\UC$, then, given any character 
$\chi_{q,n}\in \Char(\Susp)$ there is a well defined closed differential one form on $\Susp$ given by
\[ \omega_{\chi_{q,n}} := \chi_{q,n}^* d\lambda. \]

Let $X$ be the vector field tangent to the flow $\phi_t$ and let $\nu$ be any $\phi_t$--invariant Borel probability measure on $\Susp$. Define
\[ H_{f,\nu}:\Char(\Susp)\To \R \]
by
\[ H_{f,\nu}(\chi_{q,n}) := \int_{\Susp} \omega_{\chi_{q,n}}(X) d\nu \]
and observe that this definition only depends on the cohomology class of $\omega_{\chi_{q,n}}$ and the measure class of $\nu$. Hence, there is a well defined continuous homomorphism 
$\rho(f):\Char(\Susp)\To \UC$ given by
\[ \rho(f)(\chi_{q,n}) := \exp(2\pi iH_{f,\nu}(\chi_{q,n})). \]

Thus, as before,
\[ \rho(f)\in \Char(\Char(\Susp))\cong \Susp. \]

\begin{proposition} 
$\rho(f)$ is the rotation element associated to $f$ corresponding to $\nu$.
\end{proposition}

\begin{example} 
Let $\alpha$ be any element in $\Ss$ and consider the rotation $R_\alpha:\Ss\To \Ss$ given by 
$z\mTo z+\alpha$. The suspension flow $\phi_t:\Ss\times \UC\To \Ss\times \UC$ is given by
\[ \phi_t(z,s) = (z+m\alpha,t+s-m)\quad (m\leq t+s < m+1). \]

Given any character $\chi_{q,n}\in \Char(\Susp)$,
\[ \omega_{\chi_{q,n}} = qd\theta + nd\lambda \]
and the vector field $X$ associated to $\phi_t$ is constant. In this case, 
$H_{R_\alpha,\mu}(\chi_{q,n}) = \alpha q + n$ and therefore
\[ \rho(R_\alpha)(\chi_{q,n}) = \exp(2\pi iq\alpha). \] 
That is, $\rho(R_\alpha)=\alpha$ which clearly coincides with the calculation made before.
\end{example}

\section[Poincar\'e theory for compact abelian one dimensional solenoidal groups]{Poincar\'e theory for compact abelian one dimensional solenoidal groups}
\label{Poincare_theory}

This Section exhibit a demonstration of the semiconjugacy theorem for homeomorphisms of 
$\Ss$ isotopic to the identity which are pseudoirrational rotations. 

\subsection[The rotation interval]{The rotation interval}
\label{rotation_interval}

Recall that since $\Pp$ is compact and convex, and $f$ is isotopic to the identity, the image 
$\mathfrak R(\Pp)$ is a compact interval $I_f$, which, up to a translation, is contained in the one parameter subgroup $\Ll_0$. This interval is called the \textsf{rotation interval} of $f$. Since $\Ll_0$ is canonically isomorphic to $\R$, it is possible to identify $I_f$ with an interval in the real line. 

\begin{remark}
If $f$ is isotopic to an irrational rotation $R_\alpha$ with $\alpha\notin \Ll_0$ (see Section below), the rotation interval $I_f$ of $f$ can be identified with $I_f\subset \Ll_0 + \alpha$.
\end{remark}

\begin{definition} 
\label{pseudoirrational_rotation}
The homeomorphisms $f$ is a \emph{\textsf{pseudoirrational rotation}} if $I_f$ consists of a single point $I_f=\{\alpha\}$. In this case, and only in this case, we call $\alpha$ the \emph{rotation element} of $f$.
\end{definition}

\begin{remark} 
For homeomorphisms of tori $\T^n$ with $n\geq2$ the rotation set, in general, is not an interval \cite{MZ}. In our case, since the solenoids are one dimensional, the rotation set is a closed interval, which can be reduced to one point.
\end{remark}

\begin{remark} 
There are examples of \emph{diffeomorphisms} $h:\Ss\to\Ss$ (i.e. homeomorphisms which are differentiable along the leaves) such that the rotation interval is nontrivial (consists of more than one point). This is obtained by adapting the quintessential example by Katok of the torus to our case: let $X$ be the canonical unit vector field tangent to the leaves and 
$g:\Ss\To \R$ a differentiable function that vanishes at a single point $p\in \Ss$, let 
$\{h_t\}_{t\in\R}$ be the flow of the vector field $Y=gX$, then we can take $h=h_1$.
\end{remark}

\subsection[Irrational rotations]{Irrational rotations}
\label{irrational_rotations}

Since $\Ss$ is torsion free, it follows that a nontrivial rotation has no periodic points. This means the dichotomy rational / irrational does not appear in this context and we only have to define what ``irrational'' means. The following seems to be an appropriate definition: 

\begin{definition}
\label{irrational_element}
The $\alpha\in \Ss$ is \emph{\textsf{irrational}} if $\{n\alpha:n\in \Z\}$ is dense in $\Ss$. In classical terminology, $\Ss$ is said to be \emph{\textsf{monothetic}} with generator 
$\alpha$.
\end{definition}

Since $\Ss$ is a compact abelian topological group, the next theorem is classical (see e.g. \cite{Gra})

\begin{theorem}
\label{equidistribution} 
If $\alpha\in \Ss$, the following propositions are equivalent:
\begin{enumerate}
\item The rotation $R_\alpha:\Ss\To \Ss$ given by $z\mTo z+\alpha$ is ergodic with respect to the Haar measure on $\Ss$.

\item $\chi(\alpha)\neq 1$, for every nontrivial character $\chi\in \Char(\Ss)$.

\item $\Ss$ is a monothetic group with generator $\alpha$.
\end{enumerate}
\end{theorem}

\begin{remark}
\begin{enumerate}
\item Any nontrivial character $\chi\in \Char(\Ss)$ describes the solenoid $\Ss$ as a locally trivial fiber bundle over the circle $\UC$ with typical fiber a Cantor group. In fact, there is such a fibration for each $q\in \Q\setminus \{1\}$.

\item For every $\alpha\in \Ss$ and every nontrivial character, 
$\chi\circ R_\alpha = R_{\chi(\alpha)}\circ \chi$.

\item If $\alpha\in \Ss$ is irrational then $\chi(\alpha)\in \UC$ is irrational, for every nontrivial character $\chi\in \Char(\Ss)$.
\end{enumerate}
\end{remark}

\subsection[Generalized Poincar\'e theorem]{Generalized Poincar\'e theorem}

From now on we suppose that $f$ is a pseudoirrational rotation with unique rotation element$\rho(f)$.

Recall $\Ss$ is the orbit space of $\R\times \Zz$ under the $\Z$ -- action
\[ \gamma\cdot (t,k) = (t+\gamma,k-\gamma)\qquad (\gamma\in \Z). \]

Denote by $p:\R\times \Zz\To \Ss$ the canonical projection. It is clear that $p$ is an infinite cyclic covering.

If $F:\R\times \Zz\To \R\times \Zz$ is a lifting of $f$ to $\R\times \Zz$, $F$ has the form
\[ F(t,k) = (F_k(t),R_\alpha(k)), \]
where $\Zz\To \Homeo(\R)$ is a continuous function given by $k\mTo F_k$, 
$F_k:\R\To \R$ is a homeomorphism with limit periodic displacement $\Phi_t(x)$ (i.e., 
$\Phi$ is a uniform limit of periodic functions) and $\alpha\in \Zz$ is a monothetic generator. 

The condition of $F$ being equivariant with respect to the $\Z$ -- action is:
\[ F_{k-\gamma}(t+\gamma) = F_k(t) + \gamma, \] 
for any $\gamma\in \Z$. That is, $F$ commutes with the integral translation 
$T_\gamma:\R\times \Zz\To \R\times \Zz$ given by $(t,k)\mTo (t+\gamma,k)$ and also must be invariant under the $\Z$ -- action in $\Cont(\Zz,\Homeo(\R))$. 

\begin{remark}
It is very important to emphasize at this point that a lifting $F$ of $f$ exists and it is a homeomorphism of $\R\times \Zz$ due to the fact that $f$ is isotopic to the identity, which implies that $f$ keeps invariant the one dimensional leaves of the solenoid. As a consequence of this fact, $F$ leaves invariant the one dimensional leaves of $\R\times \Zz$. Since each leaf is canonically identified with $\R$, the \emph{displacement function} along the leaves can be defined in an obvious way.
\end{remark}

By the comments above it is adequate to consider $F$ as follows:
\[ F(t,k)=\left( F_k(t),k \right)\qquad (t,k)\in \R\times\Zz, \]
where $F_k:\R\To \R$ is a homeomorphism which depends continuously on $k\in \Zz$.

\begin{definition}
\label{bounded_meanvariation} 
The homeomorphism $f$ has \emph{\textsf{bounded mean variation}} if there exists $C>0$ such that the sequence $\{F_k^n(t) - t - n\tau(F)\}_{n\geq 1}$ is uniformly bounded by $C$, i.e. $C$ is independent of $(t,k)$. Here, $F_k^n$ is any lift of $f^n$ and $\tau(F)$ is a lifting of $\rho(f)$ to $\R\times \Zz$ contained in the same leaf as $(t,k)$. 
\end{definition}

\begin{remark}
Observe that if the sequence $\{F_k^n(t) - t - n\tau(F)\}_{n\geq 1}$ is uniformly bounded,  the sequence $\{F_k^n(t) - t - n\tau'\}_{n\geq 1}$ is also uniformly bounded if and only if 
$\tau'=\tau(F)$. Therefore, if a homeomorphism has bounded mean variation it is a pseudoirrational rotation.
\end{remark}

We can now state and prove the generalized version of Poincar\'e's theorem: the proof follows closely the classical proof (see \cite{Ghys}, \cite{Nav}).

\begin{theorem}
\label{Poincare_theorem} 
If $f\in \Homeo_+(\Ss)$ is a pseudoirrational rotation with irrational rotation element 
$\rho(f)$, then $f$ is semiconjugated to the irrational rotation $R_{\rho(f)}$ if and only if $f$ has bounded mean variation. 
\end{theorem}

\begin{proof}
The function $H:\R\times \Zz\To \R\times \Zz$ given by
\[ (t,k)\mTo (\sup_n \, \{F_k^n(t) - n\tau(F)\},k) \]
satisfies the following properties:
\begin{enumerate}
\item $H$ is nondecreasing, surjective and continuous on the left.

\item $H\circ T_1 = T_1\circ H$

\item $H\circ F = T_{\tau(F)}\circ H$.
\end{enumerate}

The proof that $H$ is nondecreasing follows from the fact that the lifting $F^n$ of $f^n$ is nondecreasing. The other properties on Condition (1) and Condition (2) are direct consequences of the definition of $H$ as a supremum. Condition (2) implies that $H$ descends to a map $h:\Ss\To\Ss$. Condition (3) implies that $h\circ{f}=R_{\rho(f)}\circ{h}$. Following almost \emph{verbatim} the arguments in (\cite{Ghys}, and \cite{Nav}, Theorem 2.2.6), it follows that $h$ is continuous and semiconjugates $f$ to $R_{\rho(f)}$. This follows from the fact that $F$ preserves each leaf of the form $\R\times \{k\}$, with 
$k\in \Zz$ and the map $g:\R\To \R$ given by $t\mTo \mathfrak{p}(F(t,k))-t$ is a quasihomomorphism.
\end{proof}

The immediate question arises:

\begin{question}
\label{question_conjugacy}
Under the same hypothesis of this theorem, is $f$ conjugated to the rotation $R_{\rho(f)}$ when $f$ is minimal?
\end{question}

This question and a complete development of the topological dynamics is the subject of next research (see \cite{CV}).

\section[The rotation element of a homeomorphism isotopic to a translation]{The rotation element of a homeomorphism isotopic to a translation}
\label{rotation_translation}

This Section introduces an appropriate definition of a rotation element for a homeomorphism of the one dimensional universal solenoid which is isotopic to a minimal translation by an element which is not in the base leaf. First, it is done the description of the suspension of a minimal translation in a general compact abelian group $G$, which happens to be also a compact abelian group. It follows, as a corollary, that the suspension of any homeomorphism of $G$ which is isotopic to a minimal translation is also a compact abelian group. This result is used to define the rotation element, generalizing the notion of Poincar\'e by using the theory of asymptotic cycles.
 
\subsection[The suspension of a homeomorphism isotopic to a translation]{The suspension of a homeomorphism isotopic to a translation}\label{suspension_translation}

Let $G$ be a metrizable compact abelian group and let $g\in G$. Consider the subgroup 
$\Gamma_g$ of the product group $G\times \R$ defined as follows: 
\[ \Gamma_g := \left\{ (g^n,n)\in G\times \R : n\in\Z \right\}. \]
$\Gamma_g$ is isomorphic to $\Z$, and there is a monomorphism $n\mTo (g^n,n)$. The group $\Gamma_g$ is a discrete, closed and normal subgroup of $G\times \R$. Also, 
$\Gamma_g$ is cocompact since $(G\times\R) / \Gamma_g$ is the image $p(G\times [0,1])$ of the canonical epimorphism $p : G\times\R \To (G\times \R)/\Gamma_g$.
 
\begin{remark}
Let $G$ be a metrizable compact abelian group and consider a translation 
$T:z\mTo \alpha z$. The \emph{\textsf{suspension}} of $T$ is the space 
\[ \Sigma_T(G) : = G\times \R / (x,1)\sim (T(x),0). \] 
In fact, the group $(G\times\R)/\Gamma_g$, as a topological space, is homeomorphic to the suspension of the translation $T_g:G\To G$, so we denote alternatively the group 
$(G\times\R)/\Gamma_g$ as $\Sigma_{T_g}(G)$.
\end{remark}

The theorem follows:

\begin{theorem} 
$\Sigma_T(G)$ is a compact abelian group which contains $G$ as a closed subgroup and 
\[ \Sigma_T(G)/G \cong \UC. \]
\end{theorem}

\begin{corollary} 
Since the suspension only depends on the isotopy class of the homeomorphism, the suspension of any homeomorphism $f:G\To G$ isotopic to a translation, is a compact abelian group. If the translation is a minimal translation then the suspension flow is a minimal flow (in fact the orbit through the identity is a dense one parameter subgroup).
\end{corollary}

\begin{example}
\begin{enumerate} 
\item The $2$ -- torus $\T^2$ is the suspension of an irrational rotation on the circle. 
\item The universal solenoid $\Ss$ is a suspension of a minimal translation of $\Zz$.
\end{enumerate}
\end{example}

For the particular case of the universal solenoid $\Ss$, it is not true that any homeomorphism is isotopic to a translation. In fact, in \cite{Odd} it is proved the following result:

\begin{theorem} 
If $\Homeo_\Ll(\Ss)$ is the subgroup of $\Homeo(\Ss)$ consisting of homeomorphisms of $\Ss$ that preserves the base leaf, then
\[ \Homeo(\Ss) \cong \Homeo_\Ll(\Ss) \times_\Z \Zz. \]
\end{theorem}

For instance, by Pontryagin duality, the group of automorphisms of $\Ss$ is isomorphic to the group of automorphisms of $\Q$, which is $\Q^*$, since any automorphism is determined by its value at $1$. Hence, any automorphism of $\Ss$ is never isotopic to a translation.

\begin{remark} 
The elements in the same one dimensional leaf $\Ll$ of $\Ss$ determines isotopic translations. If an element $f\in \Homeo(\Ss)$ is isotopic to a translation then $f$ is isotopic to a translation of the form $\mathfrak{t} +\gamma$, where $\gamma\in \Ll\cap \Zz$.
\end{remark} 

\subsection[The rotation element of a homeomorphism isotopic to a translation]{The rotation element with respect to an invariant measure of a homeomorphism isotopic to a translation}
\label{rotation_isotopic-translations}

Suppose that $f:\Ss\To \Ss$ is a homeomorphism which is isotopic to a minimal translation by an element not in the base leaf. According to the last section, the suspension $\Susp$ is a compact abelian group and there is a natural continuous group epimorphism 
$\Susp \To \UC$ whose kernel is a closed subgroup of $\Susp$ isomorphic to $\Ss$. Hence, there is an exact sequence of compact abelian groups 
\[ 0\To \Ss\To \Susp\To \UC \To 0. \]

By duality, there is an exact sequence of discrete groups
\[ 0\To \Z\To \Char(\Susp)\To \Q \To 0. \]

In this situation, we do not know an explicit description for $\Char(\Susp)$ and its elements. Hence, the calculation of the $1$--cocyle to describe the homomorphism 
$H_{f,\mu}\in \Hom(\Char(\Susp),\R)$ is not quite clear as we had in the isotopic to the identity case. However, knowing the fact that $\Susp$ is a compact abelian group, it is possible to calculate the values of $H_{f,\mu}$ by restricting the elements in 
$\Char(\Susp)$ to elements in $\Char(\Ss)$. Proceeding as in Section \ref{rotation_set}, this can be done in the following way (Compare \cite{Ath}).

Denote by $[z,t]$ the elements in the suspension $\Susp$ which are now equivalence classes of pairs $(z,t)$ under the suspension relation. The suspension flow is given by
\[  \phi(t,[z,s]) = [f^m(z),t+s-m], \] 
where $m\leq t+s < m+1$. As before, the canonical projection 
$\pi:\Ss\times [0,1]\To \Susp$ sends $\Ss\times \{0\}$ homeomorphically onto its image 
$\pi(\Ss\times \{0\})\equiv \Ss$ and every orbit of the suspension flow intersects $\Ss$. If $\nu$ is any $\phi_t$--invariant Borel probability measure on $\Susp$, then there is a well defined homomorphism $H_{f,\nu}:\Char(\Susp)\To \R$ given by
\[ H_{f,\nu}(\chi) = \int_{\Susp} C_{\chi}(1,[z,s]) d\nu, \]
where $C_{\chi}(1,[z,s])$ is the $1$--cocycle associated to any nontrivial character 
$\chi\in \Char(\Susp)$, at time $t=1$.

Now recall the $1$--cocycle associated to $\chi$ satisfies the relation 
(see Section \ref{rotation_set}) 
\[ C_{\chi}(t+u,[z,s]) = C_{\chi}(u,\phi_t([z,s]) + C_{\chi}(t,[z,s]), \] 
for every $t,u\in \R$ and $[z,s]\in \Susp$. Letting $s=0$, $u=t$ and $t=1$ in this relation it is obtained
\[ C_{\chi}(1+t,[z,0]) = C_{\chi}(t,[f(z),0]) + C_{\chi}(1,[z,0]). \] 

Now setting $u=1$ and $s=0$ and applying the cocycle condition on the left hand of this expression we obtain:
\[ C_{\chi}(1+t,[z,0]) = C_{\chi}(1,[z,t]) + C_{\chi}(t,[z,0]). \] 

Replacing this last equality in the first relation, rearranging the terms, and setting $t=s$, it follows that for any 
$s\in [0,1)$ and $z\in \Ss$ it holds
\[ C_\chi(1,[z,s]) = C_\chi(s,[f(z),0]) + C_\chi(1,[z,0]) - C_\chi(s,[z,0]). \]

If $\mu$ is the $f$--invariant Borel probability measure on $\Ss$ obtained by disintegration of $\nu$ with respect to the fibers, by replacing the last expression in the definition of 
$H_{f,\nu}(\chi)$ and using Fubini's theorem, we get
\begin{align*}
H_{f,\nu}(\chi)	&= \int_{\Susp} C_{\chi}(1,[z,s]) d\nu \\
&= \int_0^1 \left( \int_{\Ss} C_{\chi}(1,[z,s]) d\mu \right) ds \\
&= \int_0^1 \left( \int_{\Ss} C_{\chi}(1,[z,0]) d\mu \right) ds + 
	 \int_0^1 \left( \int_{\Ss} \left[ C_\chi(s,[f(z),0]) - C_\chi(s,[z,0])\right] d\mu \right) ds \\
&= \int_{\Ss} C_{\chi}(1,[z,0]) d\mu. 	
\end{align*}

Hence
\[ H_{f,\nu}(\chi) = \int_{\Ss} C_{\chi}(1,[z,0]) d\mu. \]

Since $\chi$ is any nontrivial character in $\Char(\Susp)$, by restricting $\chi$ to $\Ss$, one obtains a nontrivial character $\chi_q\in \Char(\Ss)$. Applying Proposition \ref{associated_cocycle} to $\chi_q$, the following relation is obtained
\[ \chi_q(f(z)) = \exp(2\pi i C_{\chi_q}(1,[z,0]) \, \chi_q(z). \]

This implies that $q(f(z)-z) - C_{\chi_q}(1,[z,0]) \in \Z$, and, since 
$q(f-\id) - C_{\chi_q}(1,[\cdot,0]))$ is a continuous function on $\Ss$, we conclude that 
$C_{\chi_q}(1,[z,0]) = q(f(z) - z)$ for any $z\in \Ss$. Since $f(z) - z = \varphi(z)$ is the displacement function along the leaves, the value of the homomorphism $H_{f,\nu}$, which now depends on $\mu$, at any character $\chi\in \Char(\Susp)$ is given by
\[ H_{f,\mu}(\chi) = q\int_\Ss \varphi d\mu. \]

Hence, for each $\mu\in \Pp$, there exists a well defined continuous homomorphism 
\[ \rho_\mu:\Char(\Susp)\To \UC \] 
given by
\begin{align*}
\rho_\mu(\chi)	&:= \exp(2\pi iH_{f,\mu}(\chi))\\
							&= \exp \left(2\pi i q\int_\Ss \varphi d\mu\right).
\end{align*}

This allows to establish the more general definition:

\begin{definition}
If $f:\Ss\To \Ss$ is any homeomorphism which is isotopic to a rotation by an element not in the base leaf, the element $\rho_\mu(f) := \rho_\mu \in \Ss$ defined as above is the \emph{\textsf{rotation element}} associated to $f$ with respect to the measure $\mu$.
\end{definition}

\begin{remark}
If $f$ is isotopic to an irrational rotation $R_\alpha$ with $\alpha\notin \Ll_0$, then the rotation interval $I_f$ of $f$ can be identified with $I_f\subset \Ll_0 + \alpha$.
\end{remark}

As indicated in the Introduction (see Section \ref{introduction}), the theory developed in this paper can be rewritten verbatim for any compact abelian one dimensional solenoidal group, since, by Pontrjagin's duality theory, any such group is the Pontryagin dual of a nontrivial additive subgroup $G\subset \Q$, where $\Q$ has the discrete topology. Denote by $\Ss_G$ such a group. According with the theory developed here, the following result is plausible:

\begin{theorem}
\label{solenoidal_Poincare-theorem}
Suppose that $f:\Ss_G\To \Ss_G$ is any homeomorphism isotopic to the identity, or isotopic to a rotation by an element not in the base leaf, with irrational rotation element $\rho(f)$. The homeomorphism $f$ is semiconjugated to the irrational rotation $R_{\rho(f)}$ if and only if 
$f$ has bounded mean variation.
\end{theorem}

The question remains:

\begin{question}
Under the same hypothesis, is $f$ is conjugated to the rotation $R_{\rho(f)}$ when $f$ is minimal? 
\end{question}

\end{document}